\documentclass[twoside,final]{siamltex}
\usepackage{amsmath}
\usepackage{mathtools}
\usepackage{amssymb}
\usepackage{amsfonts}
\usepackage{amsxtra}
\usepackage{amstext}
\usepackage{amsbsy}
\usepackage{amscd}
\usepackage{graphicx}
\usepackage{float}
\usepackage{cite}
\usepackage{xcolor}
\usepackage{srcltx} 
\usepackage{marginnote,slashbox}
\usepackage{rotating}
\usepackage{comment}
\usepackage{booktabs}            
\usepackage{algorithmic}         
\usepackage{pgfplots}
\pgfplotsset{width=6.5cm,compat=newest}
\usepackage{pgfplotstable}
\usepackage{enumitem}
\usepackage{calc}

\numberwithin{table}{section}    
\numberwithin{figure}{section}   
\numberwithin{equation}{section} 

\newcommand{\dual}[2]{\langle #1 ,\, #2 \rangle}
\renewcommand{\S}{\mathcal{S}}

\newtheorem{assumption}[theorem]{Assumption}

\begin{document}
\title{On the optimal control of some nonsmooth distributed parameter systems arising in mechanics}
\date{\today}
\author{J.C.~De Los Reyes\footnotemark[4]}
\renewcommand{\thefootnote}{\fnsymbol{footnote}}
\footnotetext[4]{Research Center on Mathematical Modelling (MODEMAT), Escuela Polit\'ecnica Nacional, Quito, Ecuador}

\maketitle
\begin{abstract}
  Variational inequalities are an important mathematical tool for modelling  free boundary problems that arise in different application areas. Due to the intricate nonsmooth structure of the resulting models, their analysis and optimization is a difficult task that has drawn the attention of researchers for several decades. In this paper we focus on a class of variational inequalities, called \emph{of the second kind}, with a twofold purpose. First, we aim at giving a glance at some of the most prominent applications of these types of variational inequalities in mechanics, and the related analytical and numerical difficulties. Second, we consider optimal control problems constrained by these variational inequalities and provide a thorough discussion on the existence of Lagrange multipliers and the different types of optimality systems that can be derived for the characterization of local minima. The article ends with a discussion of the main challenges and future perspectives of this important problem class.
\end{abstract}



\section{Introduction}
Free-boundary problems are known for their high mathematical complexity, both in terms of analysis and numerical modelling. The understanding and tracking of the free-boundary adds extra difficulties to the already challenging distributed parameter systems, and demands the use of a wide range of sofhisticated mathematical tools \cite{friedman2010variational,baiocchivariational}.


A particular class of free-boundary problems are those where a threshold behavior takes place. This means that there is a certain physical behavior if a quantity of interest remains below a given threshold value, and a different behavior if this quantity exceeds that limit. There exists a significant variety of phenomena that fit into this category. Among others we mention: viscoplastic fluids, plate deformation, frictional contact mechanics, elastoplasticity, among others.

In many circumstances, the previously described behavior can be modeled as a non-differentiable energy minimization problem of the form:
\begin{equation} \label{eq: intro: energy minimization}
  \min_{u \in V} ~E(u)=\frac{1}{2} \dual{Au}{u}+ j(u) - \dual{f}{u},
\end{equation}
where $V$ is a reflexive Banach space such that $V \hookrightarrow L^2(\Omega) \hookrightarrow V'$ with compact and continuous embedding, $A: V \mapsto V'$ is a linear elliptic operator, $\dual{\cdot}{\cdot}$ stands for the duality product between $V'$ and $V$, $f \in V'$, and $j(\cdot)$ is a convex functional of the form $$j(v)= \beta \int_{\S} |K v| ~ds, \qquad \text{with } \beta >0,$$ where $K$ stands for a linear operator, $|\cdot|$ for the Euclidean norm and $\S \subset \Omega \subset \mathbb R^d$ is the subdomain or boundary part where the nonsmooth behaviour occurs. The positive constant $\beta$
stands for the threshold coefficient, the key value for the determination of the free-boundary of the problem.

Thanks to the convexity of the energy cost functional, a necessary and sufficient optimality condition for problem \eqref{eq: intro: energy minimization} is given by what is known as a \emph{variational inequality of the second kind}: Find $u \in V$ such that
\begin{equation} \label{eq:VI}
  \dual{Au}{v-u} + \beta \int_{\S} |K v| ~ds - \beta \int_{\S} |K u| ~ds \geq \dual{f}{v-u}, \text{ for all } v \in V.
\end{equation}

Available analytical results for \eqref{eq:VI} comprise existence and uniqueness of solutions \cite{DuvautLions1976}, extra regularity of solutions \cite{Brezis1971} and, in some cases, geometric studies of the free-boundary \cite{MoMia}. Moreover, analytical results for the dynamical counterpart of \eqref{eq:VI} have also been obtained, including well-posedness and long-time behaviour of solutions \cite{DuvautLions1976}.

In practice, the application of models such as \eqref{eq:VI} requires the knowledge of the different parameters involved, especially the yield coefficient $\beta$. For estimating these quantities, a classical least-squares fitting functional with the variational inequality \eqref{eq:VI} as constraint is usually proposed.

Moreover, in some situations it is of interest not only to know the physical behaviour of the system, but also being able to act on it to achieve some predetermined objective. These types of control problems generally involve the minimization of a functional on top of the variational inequality.

Both the inverse parameter estimation problem as well as the optimal control problem present serious analytical and numerical difficulties which will be discussed in this paper by means of a model optimal control problem.

The outline of the manuscript is as follows. In Section 2 some relevant application examples are reviewed and some special properties of the problems highlighted. In Section 3 a thorough revision of the optimization results for these types of systems are summarized. Finally, in Section 4 some challenges and future perspectives are discussed.


\section{Applications}
\subsection{Viscoplastic fluid flow}
Viscoplastic fluids are characterized by the existence of a stress threshold that determines a dual behavior of the material. If the stress is below this threshold the material behaves as a rigid solid, while above this limit the behavior is that of a viscous fluid \cite{Bonn,huilgol2015fluid}.

These types of materials were investigated in the late nineteenth and early twentieth centuries by several prominent fluid mechanicists (see \cite{huilgol2015fluid} for more details). A first mathematical model was proposed by Eugene Bingham in 1922 to describe the behavior of certain suspensions. Such materials are now known precisely as Bingham fluids, in honor of the founder of rheology.

The classical steady state Bingham model considers the fluid dynamics equations with a non-differentiable term for the Cauchy tensor:
\begin{subequations} \label{eq: Bingham model}
\begin{align}
&-\textrm{Div}\,\sigma+(u\cdot \nabla) u +\nabla p= f &&\textrm{ in }\Omega, \label{eq: Bingham model1}\\
& \textrm{div}\,u =0 &&\textrm{ in }\Omega,\\
& \mathbf \sigma =2 \mu \mathcal{E} (u) + \beta
\frac{\mathcal{E}(u)}{|\mathcal{E} (u) |}, && \textrm{ if }
\mathcal{E}(u)\neq 0,
\\ &|\sigma|  \leq \beta,\hspace{0.2cm}&& \textrm{ if }\mathcal{E}(u)=0,
\end{align}
\end{subequations}
where $\mu >0$ stands for the viscosity coefficient, $\beta >0$ for the plasticity
threshold (yield stress), $f$ for the body force and Div is the
row-wise divergence operator. The deviatoric part of the Cauchy stress tensor is
denoted by $\sigma$ and $\mathcal{E}$ stands for the rate of
strain tensor, defined by
$\mathcal{E}(u):=\frac{1}{2}\left( \nabla u +\nabla u^T \right).$ The model has to be endowed with suitable boundary conditions.

The first equation in \eqref{eq: Bingham model} corresponds to the conservation of momentum, while the second one corresponds to the incompressibility condition for a fluid. As can be observed from the third equation, the Cauchy tensor is fully characterized in spatial points where the stress tensor is different from zero. If that is not the case, the material cannot be described as a fluid and the areas in which this happens are precisely known as the rigid zones (see Figure \ref{fig:driven cavity}).
\begin{figure}[H]
  \begin{center}
   \includegraphics[height=6.2cm,width=7cm]{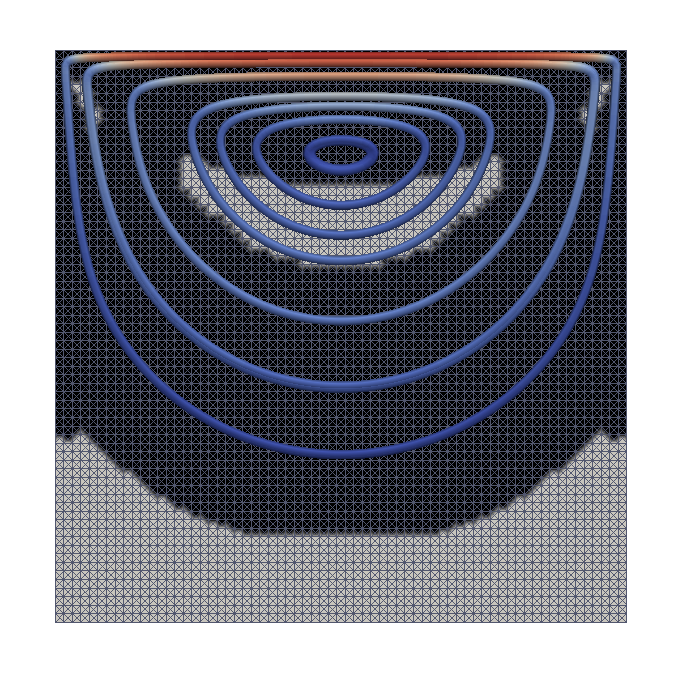}
  \caption{Steady state of a viscoplastic Bingham fluid in a wall driven cavity: Rigid zones (light gray), fluid region (black).}
  \label{fig:driven cavity}
\end{center}
\end{figure}

In the 1960s the development of functional analytic tools made it possible to study in depth a variational model for these types of materials. The resulting inequality problem consists in finding $u \in V$ such that
\begin{equation} \label{eq:probvarinvec}
\dual{Au}{v-u} + c(u,u,{v}- u)+ j(v)- j (u)\geq \langle f,v-u\rangle,\,\textrm{for all $v\in V$},
\end{equation}
where
\begin{align*}
& \dual{Au}{v}:= 2\mu\int_{\Omega}(\mathcal{E}(u):\mathcal{E}{(v)})~dx,
&& j({v}) :=\sqrt{2} \beta \int_{\Omega}{| \mathcal{E}{(v)} |} ~dx,\\
& c(u,{v},{w}) := \int_{\Omega}{w^T \left( (u\cdot
\nabla)\,{v}\right)} ~dx, &&
\end{align*}
$(C : D)=\textrm{tr}(CD^T)$, with $C,D \in \mathbb R^{d \times d}$, stands for the Frobenius scalar product and $V$ is a suitable solenoidal function space. This reformulation enabled the study of existence, uniqueness and regularity of solutions \cite{MoMia,DuvautLions1972,FuchsSeregin}. In particular, if the convective term $(u \cdot \nabla)u$ is dismissed in \eqref{eq: Bingham model1}, the resulting variational model corresponds to the necessary and sufficient optimality condition of a convex energy minimization problem as in \eqref{eq:VI}.

The numerical solution of the Bingham model has also been widely investigated. On the one hand, several discretization schemes (finite differences, finite elements, etc.) with the corresponding approximation results have been considered \cite{Glowinski,MuravlevaOlshanskii2009}. On the other hand, numerical algorithms for coping with the nonsmoothness of the underlying model have been devised. In this context we mention augmented Lagrangian methods \cite{Glowinski,AposporidisEtAl2010}, dual based algorithms (Uzawa, ISTA, FISTA, etc.) \cite{Glowinski,TRESKATIS2016115} and semismooth Newton methods \cite{dlRG2,ItoKunischBook}.


In addition to Bingham viscoplastic fluids, depending on the constitutive relation between the shear rate and the shear stress, the fluid at hand can be casted as shear thickening or shear thinning (see Figure \ref{fig: constitutive}). Important applications of these constitutive laws take place in, for instance, food industry and geophysical fluids \cite{CiarletGlowinski2010}.
\begin{figure}[H]
  \begin{center}
    \begin{tikzpicture}
    	\begin{axis}[
        height=6cm,
        width=8cm,
        xmin=0,   xmax=2.,
    		xlabel={shear rate},
    		ylabel={shear stress},
        legend style={at={(0.5,0.95)}},
    	]
      \addplot[gray,domain=0:2]{x +1};                     
      \addplot[red,domain=0:2]{1+(abs(x)^(2.5-2))*x};       
      \addplot[blue,domain=0:2]{1+(abs(x)^(1.75-2))*x};       
      \legend{{\scriptsize Bingham fluid}, {\scriptsize Shear thickening}, {\scriptsize Shear thinning}}
    	\end{axis}
    \end{tikzpicture}
    \caption{Constitutive laws for different kind of viscoplastic fluids} \label{fig: constitutive}
  \end{center}
\end{figure}
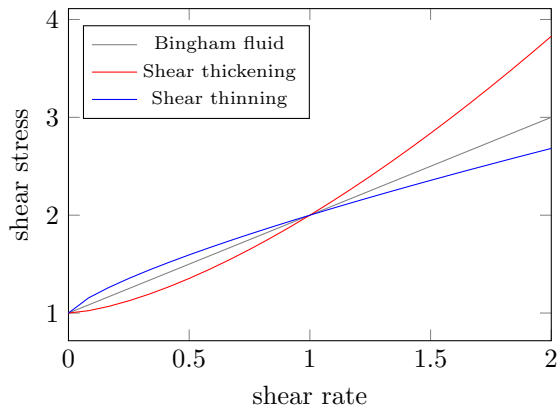

\subsection{Geophysical fluids}
A phenomenon of importance in the context of geophysics is the flow of volcanic lava. In this case the goal is to be able to numerically simulate the flow to the greatest detail, in order to predict and mitigate possible disasters (see, e.g., \cite{Fink}). As for optimization, rather than controlling the flow, the goal is to estimate the constitutive rheological parameter as best as possible, and from that, determine the explosiveness of a certain volcano.

The lava model, in addition to the complicated rheology, must be coupled with a heat transfer model that is responsible for changes in both the viscous and rheological behavior of the material. The complete model to be solved in three dimensions, in a moving domain $\Omega_t$, is given by:
\begin{subequations} \label{eq: lava full model}
\begin{align}
&\rho \frac{\partial u}{\partial t} -\textrm{Div}\,\sigma+ \rho (u\cdot \nabla) u +\nabla p= \rho f &&\textrm{ in }\Omega_t \times (0,T),\\
& \rho \frac{\partial e}{\partial t}+ \rho (u\cdot \nabla)e= k \Delta w+ \sum_{i,j} \sigma_{ij} \frac{\partial u_i}{\partial x_j} &&\textrm{ in }\Omega_t \times (0,T),\\
& \textrm{div}\,u =0 &&\textrm{ in }\Omega_t \times (0,T),\\
& \mathbf \sigma =2 \mu \mathcal{E} (u) +\beta
\frac{\mathcal{E}(u)}{|\mathcal{E} (u) |}, && \textrm{ if }
\mathcal{E}(u)\neq 0,
\\ &|\sigma|  \leq \beta,\hspace{0.2cm}&& \textrm{ if }\mathcal{E}(u)=0,
\end{align}
\end{subequations}
complemented with suitable initial and boundary conditions. In addition to the previous notation, $\rho$ stands for the fluid density, $e$ for the enthalpy function, $w$ is the temperature, $k$ is the thermal conductivity and $f$ is the gravity force. In the simplest case, the enthalpy is given by the product between a constant specific heat and the temperature (see, e.g., \cite{costa2005computational} for further details).

In general, scarce analytical and numerical investigation has been carried out for model \eqref{eq: lava full model}. Due to its high complexity, questions about existence and uniqueness, solution regularity and properties of the solution operator are still open. Moreover, its numerical treatment posses several challenges concerning discretization, solution algorithms, tracking of the interface and computational efficiency.

A common approach, among computational geophysicists, to deal with \eqref{eq: lava full model} consists in using approximations that reduce the dimensionality of the problem and allow to solve it computationally in a more straightforward way. A very popular technique is the shallow water approximation \cite{bernabeu2016modelling}, which enables a faster numerical solution of the problem without dismissing the topography of the terrain, a crucial variable in this type of phenomena. The price to pay for this simplification is the obtention of a system of hyperbolic conservation laws, which demand sophisticated techniques for their analysis and numerical solution.

\subsection{Elastoplasticity}
The transition from an elastic to a plastic regime in a solid body subject to different loads is a natural candidate phenomena to be modeled with variational inequalities, since this transition occurs precisely when the stress exceeds a certain threshold.

One of the interesting mathematical properties of these types of phenomena is that they can be variationally formulated in terms of their primal or their dual variables \cite{temamplast,han2012plasticity}. In the first case, a variational inequality of the second kind is obtained, while in the other, an obstacle-like inequality is derived. This fact has been much utilized in the analysis and numerical simulation of this type of problems.

Under the assumption of small strains, a quasi-static linear kinematic hardening model with a von Mises yield condition in its \emph{primal formulation} is given in the following way. Consider a solid body $\Omega \in \mathbb R^3$ that is clamped on a nonvanishing Dirichlet part $\Gamma_D$ of its boundary $\Gamma$, and it is subject to boundary loads on the remaining Neumann part $\Gamma_F$. The variables of the problem are the \emph{displacement} $u \in V:=H_{\Gamma_D}^1(\Omega;\mathbb R^3)$ and the \emph{plastic strain} $q \in Q :=\{ p \in L^2(\Omega; \mathbb R^{3 \times 3}_{sym}): \text{trace}(p)=0 \}$, with $\mathbb{R}^{3 \times 3}_{sym}$ the space of symmetric matrices. The problem consists in finding $W=(u,p)$ which satisfies
\begin{equation*}
  \dual{A W}{Y-W}_{Z',Z} + j(q)-j(p) \geq \dual{f}{v-u}, \text{ for all } Y=(v,q) \in Z=V \times Q,
\end{equation*}
where
\begin{align*}
  &\dual{A W}{Y}_{Z',Z}= \int_\Omega \left( \mathcal{E}(u)-p: \mathbb C(\mathcal{E}(v)-q) \right) ~dx + \int_\Omega (p: \mathbb H q) ~dx,\\
  &j(p)= \beta \int_\Omega |p| ~dx,\\
  &\dual{f}{v}= \int_\Omega l v ~dx + \int_{\Gamma_F} g \cdot v ~ds,
\end{align*}
$\mathbb C$ represents the material's fourth-order elasticity tensor and $\mathbb H$ is the hardening modulus. The constant $\beta > 0$ denotes the material's yield stress and the data $l$ and $g$ are the volume and boundary loads, respectively.

If the temperature is also taken into account, the phenomenon gains in complexity and its modeling is possible only through the use of primal variables, that is, through variational inequalities of the second kind \cite{ottosen2005mechanics}.
At present, this phenomenon is intensively investigated in terms of the analysis of solutions \cite{bartels2008thermoviscoplasticity,chelminski2006mathematical} and their numerical approximation \cite{bartels2013numerical}.

Optimal control problems in elastoplasticity have also been studied in recent years, yielding optimality conditions for its primal and dual variants \cite{herzog2012c,HerzogMeyerWachsmuth,dlRHM13}, as well as numerical algorithms for solving the problem \cite{herzog2014optimal}. The thermoviscoplastic case is also currently being addressed with very challenging and promising perspectives \cite{herzog2015existence}.

\subsection{Contact mechanics}
One of the most emblematic problems modeled by variational inequalities occurs in contact mechanics and is known as Signorini's problem. It consists in determining the deformation of a certain surface subject to external forces, and in contact with an obstacle. This last fact induces forces in the normal direction to the contact surface, modeled by variational inequalities of the first kind, and tangential forces (friction) along the contact region, usually modeled by inequalities of the second kind. This combination of phenomena, however, occurs in a non-linear fashion, leading to a more complicated category of inequalities known as quasi-variational.

Contact problems have been widely addressed in the literature. The first works of Heinrich Hertz in the nineteenth century were followed by several contributions on the modelling of the problems, the analysis of existence and uniqueness of solutions \cite{EkeTemam,Stampacchia/Kinderlehrer,SofoneaMatei2009}, the regularity of solutions and of the free-boundary \cite{Stampacchia/Kinderlehrer,caffarelli2005geometric}, and the numerical approximation and solution of the models \cite{Glowinski,KikuchiOden,wriggers2006computational}.

A frequently used version of Signorini's contact problem is the one with so-called Coulomb friction law. For its formulation, let us consider $\Omega \subset \mathbb R^d$, $d=2,3,$ a bounded domain with regular boundary
$\Gamma$. The boundary can be further divided into three non-intersecting
components $\Gamma=\Gamma_D \uplus \Gamma_F \uplus \Gamma_C$, corresponding to the
Dirichlet, Neumann and contact boundary sectors, respectively. The friction forces
intervene only on the boundary sector where contact with the rigid foundation
takes place. The problem consists in finding a displacement vector $u$ that solves the following system:
\begin{subequations} \label{eq: Signorini with Coulomb}
  \begin{align}
    &- \textrm{Div} \, \sigma =f_1 &&\text{ in }\Omega,\\
    & u =0 &&\text{ on }\Gamma_D,\\
    & \sigma_N (u)  =t &&\text{ on }\Gamma_F\\
    & u_N \leq g, \quad \sigma_N (u) \leq 0, \quad (u_N-g) \, \sigma_N (u)=0 &&\textrm{ on }\Gamma_C,\\
    & \sigma_T (u)= - \beta(u) \frac{u_T}{|u_T |}, && \textrm{ on } \{ x
    \in \Gamma_C: u_T \neq 0 \}, \label{eq:general friction law 1}\\
    &|\sigma_T (u)|  \leq \beta(u), && \textrm{ on } \{ x \in \Gamma_C: u_T = 0
    \},\label{eq:general friction law 2}
  \end{align}
\end{subequations}
where $g$ denotes the gap between the bodies and $N$ and $T$ stand for the unit outward normal and unit tangential vector, respectively. The notation $u_N$ and $u_T$ stands for the product $u \cdot N$ and $u \cdot T$, respectively.
The stress-strain relation for a linear elastic material is given by Hooke’s law:
\begin{equation}
  \sigma=  2 \mu \mathcal E(u) + \lambda ~\text{tr}(\mathcal E(u)),
\end{equation}
where $\mathcal E(u):=\frac{1}{2}\left( \nabla u +\nabla u^T \right)$ stands for the rate of strain tensor and $\lambda>0$ and $\mu >0$ are the Lam\'e parameters.

If the friction effect is dismissed, the resulting model may be formulated as a variational inequality of the first kind: Find $u \in \mathcal K := \{ v \in H^1_{\Gamma_D}(\Omega): v_N \leq g \text{ a.e. on } \Gamma_C \}$ such that
\begin{equation} \label{eq: contact VI first}
  \dual{Au}{v-u} \geq \dual{f}{v-u}, \text{ for all }v \in \mathcal K,
\end{equation}
where
\begin{equation*}
 \dual{Au}{v}:= \int_{\Omega}(\sigma :\mathcal{E}{v})~dx, \quad  \quad
 \dual{f}{v} :=\int_{\Omega}{f_1 \,v}~dx+ \int_{\Gamma_F} t \, v ~ds
\end{equation*}
and $H^1_{\Gamma_D}(\Omega):= \{ v \in H^1(\Omega): v=0 \text{ on }\Gamma_D \}$.

If, on the contrary, the contact surface is assumed to be known and only the nonsmoothness due to friction is taken into account, the phenomenon is modeled by a variational inequality of the second kind:
\begin{equation} \label{eq: contact VI second}
  \dual{Au}{v-u} +\beta \int_{\Gamma_F}|v| ~ds- \beta \int_{\Gamma_F}{|u|}~ds \geq \dual{f}{v-u}, \text{ for all }v \in V:= H^1_{\Gamma_D}(\Omega).
\end{equation}


Optimal control and inverse problems in contact mechanics have been addressed in, e.g., \cite{BergouniouxMignot2000,BeremlijskiEtAl2002,BermudezSaguez1987,jarusekoutrata2007,betz2015optimal}, generally using the simplified versions \eqref{eq: contact VI first} or \eqref{eq: contact VI second} of Signorini's problem. Optimality conditions of more or less sharpness are currently available for these types of problems, which have been derived using similar techniques to those that will be explained in the forthcoming Section \ref{sec: optimal}.
In the case of the complete model the problem is still farely open. The complexity of the combined nonlinearities and nonsmoothness makes the analysis extremely imbricated. Some initial attempts to deal with such structures have been carried out in \cite{dietrich2001optimal}.


\section{Optimal control} \label{sec: optimal}
The development of optimal control theory is closely linked to the space race in the twentieth century. The moon landing problem is a classical application example of this theory, where the trajectory, velocity and acceleration of a space vehicle had to be optimally determined in each instant of time to achieve a desired goal. The celebrated Pontryagin maximum principle \cite{pontryagin1987mathematical} was, within this framework, a milestone that made it possible to actually solve the resulting control problems.

The extension of the theory to models with partial differential equations started to take place in the 1960s. The main goal in this case was to extend and develop techniques to cope with cases where the spatial variable also played a crucial role, which occurs, for instance, in diffusion processes. At present there are important established techniques for the mathematical analysis and numerical solution of such PDE control problems \cite{Li1,troltzsch2010optimal,de2015numerical}.

Further, optimal control problems governed by variational inequalities are related to the design of mechanisms to act on the dynamics of a nonsmooth distributed parameter system to guide them towards some desired tarjet. Such problems where considered since the 1970s, with a renewed interest in the field in recent years, due to the wide applicability of the results.

In addition to control, some relevant and related problems take place when trying to estimate different coefficients involved in the distributed parameter system. In the case of the variational inequalities under consideration, estimating the threshold coefficient by solving the resulting inverse problem appears to be of high relevance for the understanding of the material at hand.

In this section we present the main up-to-date results related to the optimal control of variational inequalities of the second kind by means of the following model problem: Find an optimal control $f \in L^2(\Omega)$ and a corresponding state $u \in V$ solution of
\begin{equation} \label{eq: optimal control problem} \tag{$\mathcal P$}
\left\{
\begin{array}{ll}
\min \limits_{(u,f) \in V \times L^2(\Omega)} ~J(u,f)=\frac{1}{2} \int_\Omega |u-z_d|^2 ~dx + \frac{\alpha}{2} \int_\Omega |f|^2 ~dx \vspace{0.3cm}\\
\text{subject to:  }\\[3pt]
\quad \dual{Au}{v-u}+\beta \int_\Omega |v|~dx -\beta \int_\Omega |u|~dx \geq  \dual{f}{v -u},  &\text{ for all } v \in V.
\end{array}
\right.
\end{equation}
where $\alpha >0$, $\Omega \subset \mathbb R^d$ is a bounded domain and $z_d \in L^2(\Omega)$. We recall that $V$ is assumed to be a reflexive Banach space such that $V \hookrightarrow L^2(\Omega) \hookrightarrow V'$ with compact and continuous embedding, $A: V \mapsto V'$ is a linear elliptic operator and $\dual{\cdot}{\cdot}$ stands for the duality product between $V'$ and $V$. Although more general cost functionals may be considered as well, we restrict our attention here to the tracking type one which is the first and more intuitive choice.

Using the direct method of the calculus of variations \cite{de2015numerical} it can be proved that there exists a unique solution to the lower level problem in \eqref{eq: optimal control problem}. Moreover, by duality arguments \cite{EkeTemam}, there exists a dual multiplier $q \in L^{\infty}(\Omega)$ such that the following primal-dual system holds:
\begin{equation}
\begin{array}{ll}
\dual{Au}{v}+\int_\Omega q \, v~dx = \dual{f}{v},  &\text{ for all } v \in V, \\
[3pt]
q(x) u(x) = \beta |u(x)|,  &\text{ a.e. in } \Omega,\\
[3pt]
|q(x)| \leq \beta,  &\text{ a.e. in } \Omega.
\end{array}
\end{equation}
The presence of the dual variable is not only important theoretically, but also numerically, since it gives rise to important dual and primal-dual solution algorithms \cite{Glowinski}.

With help of primal and dual variables we may define the active, inactive and biactive sets for the problem as follows:
\begin{align*}
  \mathcal A &= \{ x\in \Omega: u(x) = 0\} && \text{(active set)},\\
  \mathcal I &= \{ x\in \Omega: u(x) \neq 0\} && \text{(inactive set)},\\
  \mathcal B &= \{ x\in \Omega: u(x) = 0 \, \land \, |q(x)|=\beta \} && \text{(biactive set)} .
\end{align*}

Next we will show that there exists a global optimal solution to problem \eqref{eq: optimal control problem}, which is, however, not necessarily unique. Although in practice it may not be possible to compute global but only local minima, the next global existence result constitutes the first step towards the successful analysis of the optimal control problem at hand.
\begin{theorem}
  Problem \eqref{eq: optimal control problem} has at least one optimal solution.
\end{theorem}
\begin{proof}
Since the cost functional is bounded from below, there exists a minimizing sequence $\{ (u_n,f_n) \}$, i.e., $J(u_n,f_n) \to \inf_{f} J(u,f),$
where $u_n$ stands for the unique solution to
\begin{equation} \label{eq:NS-VI-3}
\dual{Au_n}{v-u_n}+\beta \int_\Omega |v|~dx -\beta \int_\Omega |u_n|~dx = \dual{f_n}{v-u_n},  \text{ for all } v \in V.
\end{equation}
From the structure of the cost functional it also follows that $\{ f_n \}$ is
bounded in $L^2(\Omega)$ and, thanks to \eqref{eq:NS-VI-3}, also
$\{u_n\}$ is bounded in $V$.
Consequently, there exists a subsequence (denoted in the same way) such that
\begin{equation*}
    f_n \rightharpoonup \hat f \text{ weakly in } L^2(\Omega) \hspace{0.5cm}\text{ and } \hspace{0.5cm} u_n \rightharpoonup \hat u \text{ weakly in }V.
\end{equation*}
Due to the compact embedding $L^2(\Omega) \hookrightarrow
V'$ it then follows that
\begin{equation*}
    u_n \to \hat u \text{ strongly in } V'.
\end{equation*}

From \eqref{eq:NS-VI-3} we directly obtain that
\begin{equation*}
 \dual{Au_n}{u_n}-\dual{Au_n}{v} +j(u_n)-j(v) - \langle f_n ,u_n- v \rangle \leq 0, ~\forall v \in V.
\end{equation*}
Thanks to the convexity and continuity of $\dual{A \cdot}{\cdot}$ and $j(\cdot)$ we may take the limit inferior in the previous inequality and obtain that
\begin{equation}
\dual{A \hat u}{\hat u}-\dual{A \hat u}{v}+ j(\hat u)-j(v) - \langle \hat f ,\hat u- v \rangle \leq 0, ~\forall v \in V,
\end{equation}
which implies that $\hat u$ solves the lower level problem with $\hat f$ on the right hand side.

Thanks to the weakly lower semicontinuity of the cost functional we finally obtain that
\begin{equation*}
    J(\hat u,\hat f) \leq \liminf_{n \to \infty} J(u(f_n),f_n) = \inf_{f} J(u(f),f),
\end{equation*}
which implies the result.
\end{proof}

Once the existence of optimal solutions is guaranteed, the next step consists in characterizing local optima by means of first order optimality conditions, also called optimality systems.

\subsection{Optimality systems}
As in finite dimensions, it is in general not possible to verify standard constraint qualification conditions for infinite dimensional nonsmooth optimization problems like \eqref{eq: optimal control problem}. Consequently, in order to get a Karush-Kuhn-Tucker optimality system, alternative techniques have to be devised.

One of the possibilities to derive optimality conditions for problem \eqref{eq: optimal control problem} consists in regularizing the non-differentiable term, getting rid of the nonsmoothness. In \cite{Barbu1984}, for instance, a general regularization procedure is presented, where the functional $j(\cdot)$ is replaced by a smooth approximation of it. Specifically, global type regularizations like
$$\phi_{\gamma}(x):= \beta \sqrt{|x|^2+ \gamma^{-2}} \quad \text{ or } \quad \phi_{\gamma}(x):=\frac{\beta}{\gamma + 1} \left(\gamma |x| \right)^{\frac{\gamma+1}{\gamma}},$$
with $\gamma >0$, were frequently used. The resulting regularized control problems can be analyzed using PDE-constrained optimization techniques, yielding the following first order optimality system of Karush-Kuhn-Tucker type:
\begin{subequations} \label{eq: reg. OS}
\begin{align}
& \dual{Au_\gamma}{v}+\int_\Omega q_\gamma \, v~dx = \dual{f_\gamma}{v},&& \text{ for all } v \in V, \label{eq: reg. OS1}\\
& q_\gamma(x) = \phi_\gamma'(x) && \text{ a.e. in }\Omega, \label{eq: reg. OS2}\\
& \dual{A^* p_\gamma}{v}+\int_\Omega \phi_\gamma''(u_\gamma)^*p_\gamma \, v~dx = \int_\Omega (u_\gamma-z_d) \,v ~dx,&& \text{ for all } v \in V,\label{eq: reg. OS3}\\
& \alpha f_\gamma+p_\gamma = 0&& \text{ a.e. in }\Omega \label{eq: reg. OS4}
\end{align}
\end{subequations}

Concerning the consistency of the regularization, usually two types of results shall be proved. The first one guarantees that the family of regularized controls $\{ f_\gamma \}_{\gamma >0}$ contains a weakly convergent subsequence whose limit precisely solves problem \eqref{eq: optimal control problem}. The other type of consistency result assures that, given a local optimal control which satisfies a quadratic growth condition, there is a family of regularized controls that approximate it.

Once the consistency has been analyzed, an optimality system for the original problem can be obtained by passing to the limit in system \eqref{eq: reg. OS} (see, e.g., \cite{Barbu1984,BonnansTiba1991}).
\begin{theorem}
  Let $f \in L^2(\Omega)$ be a local optimal solution of \eqref{eq: optimal control problem} and $u \in V$ its corresponding state. Let $\{ f_\gamma \}_{\gamma >0}$ be a
  sequence of regularized optimal controls such that
  $f_\gamma \rightharpoonup f$ weakly in $L^2(\Omega)$, as $\gamma \to \infty$. Then there exist multipliers $p \in V$ and $\xi \in V'$ such that the following system holds:
  \begin{subequations} \label{eq: OSweak}
  \begin{align}
  & \dual{Au}{v}+\int_\Omega q \, v~dx = \dual{f}{v},&& \text{ for all } v \in V, \label{eq: OSweak1}\\
  & q(x)  u(x) =\beta | u(x)| && \text{ a.e. in }\Omega, \label{eq: OSweak2}\\
  & |q(x)| \leq \beta && \text{ a.e. in }\Omega, \label{eq: OSweak3}\\
  & \dual{A^* p}{v}+\langle \xi, v \rangle = \int_\Omega (u-z_d) \,v ~dx,&& \text{ for all } v \in V \label{eq: OSweak4}\\
  & \alpha f+p=0&& \text{ a.e. in }\Omega. \label{eq: OSweak5}
  \end{align}
  \end{subequations}
\end{theorem}
\begin{proof}
  Equations \eqref{eq: OSweak1}-\eqref{eq: OSweak2} are
  obtained directly from the continuity of the regularized solution operator.
  Testing \eqref{eq: reg. OS3} with $v=p_\gamma$ yields
  \begin{equation} \label{eq:adjoint eq multiplied by p}
   \dual{A^* p_\gamma}{p_\gamma}+\int_\Omega p_\gamma \, \phi_\gamma''(u_\gamma)^* p_\gamma~dx = \int_\Omega (u_\gamma-z_d) \,p_\gamma ~dx
  \end{equation}
Thanks to the convexity of the regularizing function $\phi_\gamma$, it follows that $$\int_\Omega p_\gamma \phi_\gamma''(u_\gamma)^* p_\gamma~dx \geq 0,$$ which together with the ellipticity of the operator $A$ and the boundedness of the sequence $\{u_\gamma\}_{\gamma >0}$ implies that
\begin{equation} \label{eq:bound of reg. seq. of adjoints}
 \|p_\gamma \|_V \leq C_p, \quad \text{ for all }\gamma>0,
\end{equation}
i.e., the sequence $\{ p_\gamma \}_{\gamma >0}$ is bounded in $V$ and there exists a subsequence
(denoted in the same way) and a limit $p \in V$ such that
\begin{align*}
p_\gamma \rightharpoonup p  \text{ weakly in }V \qquad \text{and} \qquad
A p_\gamma \rightharpoonup A p  \text{ weakly in }V'.
\end{align*}

From the the latter and the boundedness of $\{u_\gamma\}_{\gamma >0}$, we obtain that $\{ \phi_\gamma''(u_\gamma)^*p_\gamma \}_{\gamma >0}$ is bounded in $V'$. Consequently there exists a
subsequence (denoted the same) and a limit $\xi \in V'$ such
that
$$\phi_\gamma''(u_\gamma)^*p_\gamma \rightharpoonup \xi \text{ weakly in } V'.$$
Passing to the limit in \eqref{eq: reg. OS3} and \eqref{eq: reg. OS4} then yields the result.
\end{proof}

Although system \eqref{eq: OSweak} includes equation \eqref{eq: OSweak4} for the adjoint state $p$, it does not characterize its behavior in relation to the state $u$, the dual multiplier $q$ or the additional multiplier $\xi$. This is a main drawback which makes the characterization incomplete, i.e., there may exist several solutions of system \eqref{eq: OSweak} that do not correspond to stationary points of the optimal control problem.

Through the use of tailored local regularizations, more detailed optimality systems can be obtained \cite{Delosreyes2009}. In particular, relations between the quantities $u,~q, ~p$ and $\xi$ are obtained within the optimality condition, ressembling what is known as Clarke-stationarity in finite-dimensional nonsmooth optimization \cite{sun2006optimization}. Such tailored regularizations seek to locally approximate the generalized derivative of the Euclidean norm, that is, the regularized function coincides exactly with the generalized derivative, except in a neighborhood of the non-differentiable elements. In particular, the following smoothing function
\begin{equation}
\phi_{\gamma}'(x)=
\begin{cases}
\beta \frac{x}{|x|} &\text{ if }~\gamma |x| \geq \beta + \frac{1}{2\gamma},\\
 \frac{x}{|x|} (\beta- \frac{\gamma}{2} (\beta- \gamma |x|+\frac{1}{2\gamma})^2) &\text{ if }~\beta-\frac{1}{2\gamma}\leq \gamma |x| \leq \beta+\frac{1}{2\gamma},\\
\gamma x &\text{ if }~\gamma |x| \leq \beta-\frac{1}{2\gamma},
\end{cases}
\end{equation}
for $\gamma$ sufficiently large, has been proposed, yielding the following result \cite{Delosreyes2009}.

\begin{theorem}
  Let $f \in L^2(\Omega)$ be a local optimal solution of \eqref{eq: optimal control problem} and $u \in V$ its corresponding state. Let $\{ f_\gamma \}_{\gamma >0}$ be a
  sequence of (locally) regularized optimal controls such that
  $f_\gamma \rightharpoonup f$ weakly in $L^2(\Omega)$, as $\gamma \to \infty$. Then there exist multipliers $p \in V$ and $\xi \in V'$ such that the following system holds:
  \begin{subequations} \label{eq: OS with local regularization}
  \begin{align}
  & \dual{Au}{v}+ \int_\Omega q \, v~dx = \dual{f}{v},&& \text{ for all } v \in V,\\
  & q(x)  u(x) =\beta | u(x)| && \text{ a.e. in }\Omega,\\
  & |q(x)| \leq \beta && \text{ a.e. in }\Omega,\\
  & \dual{A^* p}{v}+\langle \xi, v \rangle = \int_\Omega (u-z_d) \,v ~dx,&& \text{ for all } v \in V,\\
  & \alpha f+p=0&& \text{ a.e. in }\Omega,
  \end{align}
  and, additionally,
  \begin{align} \label{eq: compl in OS with local reg}
    & p(x)=0 && \text{ a.e. in }\mathcal I:=\{ x: |q(x)| < \beta \},\\
& \dual{\xi}{p} \geq 0,\\
   & \dual{\xi}{u}=0.
 \end{align}
\end{subequations}
\end{theorem}


In addition to the complementarity relations, system \eqref{eq: OS with local regularization} has as main advantage the fact that it can be derived for different types of controls (distributed, boundary, coefficients) and in the presence of additional control, mixed or state constraints.


Alternatively, in order to derive a stronger optimality system, the nonsmooth properties of the control-to-state operator, including some sort of differentiability, have to be carefully analyzed (see, e.g., \cite{Mignot1976,MignotPuel1984,DelosReyesMeyer2015}). In the next result we show that such solution operator satisfies a Lipschitz property.

\begin{lemma}\label{lem:lipschitz}
 For every $f\in V'$ there exists a unique solution $u\in V$ of
\begin{equation} \label{eq: VI Lipschitz}
  \dual{Au}{v-u}+\beta \int_\Omega |v|~dx -\beta \int_\Omega |u|~dx \geq \dual{f}{v -u},  \quad \text{ for all } v \in V,
\end{equation}
 which we denote by $u = S(f)$. The associated solution operator $S: V' \to V$ is globally Lipschitz continuous, i.e., there exists a constant $L > 0$
 such that
 \begin{equation}
  \|S(f_1) - S(f_2)\|_V \leq L \, \|f_1 - f_2\|_{V'} \quad \forall \, f_1, f_2 \in V'.
 \end{equation}
\end{lemma}
\begin{proof}
 Existence and uniqueness follows by standard arguments from the maximal monotonicity of $A + \partial \| \cdot \|_{L^1(\Omega)}$,
 see for instance \cite{Barbu1993}.
 To prove the Lipschitz continuity we test the variational inequality \eqref{eq: VI Lipschitz} for $u_1 = S(f_1)$ with $u_2 = S(f_2)$ and vice versa and add the arising
 inequalities to obtain
 \begin{equation*}
  \dual{A(u_1 - u_2)}{u_1 - u_2} \leq \dual{f_1 - f_2}{u_1 - u_2}.
 \end{equation*}
 The ellipticity of $A$ then yields the result.
\end{proof}

In addition to the Lipschitz continuity, the directional differentiability of the solution operator is indispensable in order to derive stronger optimality conditions. The obtention of such a result, however, requires some additional assumptions about the structure of the biactive set and the regularity of the solution.

\begin{assumption} \label{assu: structural}
  The active set $\mathcal A = \{ x\in \Omega: u(x) = 0\}$ satisfies the following conditions:
  \begin{enumerate}
   \item\label{assu:active1} $\mathcal A = \mathcal A_1 \cup \mathcal A_0$, where $\mathcal A_1$ has positive measure and $\mathcal A_0$ has zero capacity \cite{attouch2014variational}.
   \item\label{assu:active2} $\mathcal A_1$ is closed with non-empty interior.
   Moreover, it holds $\mathcal A_1 = \overline{\text{int}(\mathcal A_1)}$.
   \item\label{assu:active3} For the set $\mathcal J:= \Omega\setminus \mathcal A_1$ it holds
   \begin{equation}\label{eq:innererrand}
    \partial\mathcal J \setminus (\partial\mathcal J\cap\partial\Omega) = \partial\mathcal A_1\setminus(\partial\mathcal A_1\cap\partial\Omega),
   \end{equation}
   and both $\mathcal A_1$ and $\mathcal J$ are supposed to have regular boundaries.
   That is, the connected components of $\mathcal J$ and $\mathcal A_1$ have positive distance from each other and
   the boundaries of each of them satisfies the cone condition \cite{gri85}.
  \end{enumerate}
\end{assumption}

  \begin{figure}
    \begin{center}
    \includegraphics[height=4.5cm]{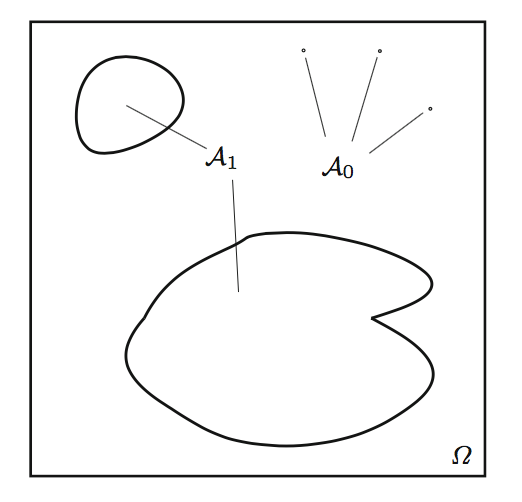} \includegraphics[height=4.5cm]{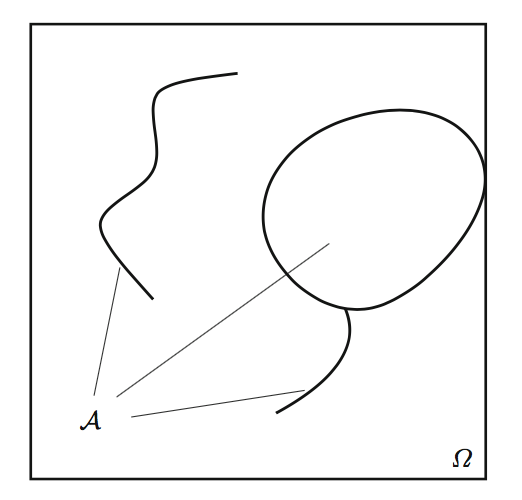}
    \caption{Allowed (left) and not allowed (right) active sets according to Assumption \ref{assu: structural}}
    \label{fig: active}
    \end{center}
  \end{figure}

Assumption \ref{assu: structural} has been relaxed in \cite{christofmeyer2017} allowing for the presence of $d-1$ dimensional active subsets. An alternative polyhedricity hypothesis has been considered recently in \cite{hintermuller2017directional}. Although the latter apparently avoids structural assumptions on the active set, in the recent work \cite{christof2017non} it is shown that in order to get polyhedricity, structural assumptions on the active set are actually unavoidable.

\begin{theorem}\label{thm:ablvi}
 Let $f,h \in L^r(\Omega)$ with $r > \max\{d/2,1\}$ be given. Suppose further that Assumption \ref{assu: structural}
is fulfilled by $u = S(f)$ and the associated slack variable $q$, and that both functions are continuous. Then there holds
 \begin{equation}\label{eq:weaklim}
  \frac{S(f + t\,h) - S(f)}{t} \rightharpoonup \eta \quad \text{weak in } V, \quad \text{as } t \searrow 0,
 \end{equation}
 where $\eta \in V$ solves the following VI of first kind:
 \begin{equation}\label{eq:ablvi}
 \begin{aligned}
  \eta \in \mathcal K(u),\quad \dual{A\eta}{v-\eta} \geq \dual{h}{v-\eta}, \quad \forall\, v\in \mathcal K(u),
 \end{aligned}
 \end{equation}
 with $$\mathcal K(u)=\begin{aligned}[t]
  \{v\in V: \;\, & v(x) = 0 \text{ a.e., where } |q(x)| < \beta,\\
  & v(x)q(x) \geq 0 \text{ a.e., where } |q(x)| = \beta \text{ and } u(x) = 0\}.
 \end{aligned}$$
\end{theorem}

The last theorem establishes a directional differentiability result for the solution operator in a weak sense. Composed with the quadratic structure of the tracking type cost functional, the directional differentiability of the reduced cost is obtained. In case $\mathcal B = \emptyset,$ the result can be further improved and G\^ateaux differentiability is obtained.

Theorem \ref{thm:ablvi} was recently generalized in \cite{christofmeyer2017}, where in addition to improving Assumption \ref{assu: structural}, semilinear terms in the inequality were considered.

Similarly as for optimal control of obstacle problems (see \cite{MignotPuel1984}), a stronger stationarity condition can only be obtained if the control is of distributed type and no control constraints are imposed (see \cite{wachsmuth2014strong} for further details on the presence of control constraints). In the following theorem a strong stationary optimality system is established for the model optimal control problem \eqref{eq: optimal control problem}.


\begin{theorem}
  Let $f \in L^2(\Omega)$ be a local optimal solution of \eqref{eq: optimal control problem} and $u \in V$ its corresponding state. Suppose that Assumption \ref{assu: structural} holds and assume further that $u,q \in C(\bar{\Omega})$. Then there exist multipliers $p \in V$ and $\xi \in V'$ such that the following system holds:
  \begin{subequations} \label{eq: strong stationary OS}
  \begin{align}
  & \dual{Au}{v}+ \int_\Omega q \, v~dx = \dual{f}{v},&& \text{ for all } v \in V,\\
  & q(x)  u(x) =\beta | u(x)| && \text{ a.e. in }\Omega,\\
  & |q(x)| \leq \beta && \text{ a.e. in }\Omega,\\
  & \dual{A^* p}{v} +\langle \xi, v \rangle = \int_\Omega (u-z_d) \,v ~dx,&& \text{ for all } v \in V,\\
  & \alpha f+p=0&& \text{ a.e. in }\Omega,
\end{align}
  and, additionally,
  \begin{align}
  & p(x)=0 &&\text{a.e. in }\mathcal I:=\{ x: |q(x)| < \beta \},\\
  &p(x) q(x)=0 && \text{a.e. in } \mathcal B,\\
   &\dual{\xi}{v} \geq 0, && \forall v \in V: v(x)=0 \text{ if } |q(x)|< \beta \land v(x) q(x) \geq 0 \text{ a.e. in } \mathcal B,
\end{align}
\end{subequations}
 where $\mathcal B := \{x \in \Omega: u(x)=0 \land |q(x)|=\beta \}$.
\end{theorem}

Optimality system \eqref{eq: strong stationary OS} is sharper than \eqref{eq: OS with local regularization} since it includes a pointwise relation between the adjoint state and the dual multiplier on the biactive set and a sign condition on the additional multiplier $\xi$ for a specific test function set. However, it is worth noting that the required assumptions for getting this result are quite strong and not extendable to several cases of interest, like boundary control or control constrained problems.

Finally, let us recall that in case $\mathcal B = \emptyset,$ the problem becomes smooth and optimality systems \eqref{eq: OS with local regularization} and \eqref{eq: strong stationary OS} coincide.


\section{Challenges and perspectives}
Although several efforts have been made in the study of \emph{variational inequalities of the second kind} and their \emph{optimal control}, the topic is still very active and full of challenges. We comment next on some extensions, open problems and future perspectives within this field.


\subsection{Different operators $K$}
The model optimal control problem previously considered deals with the $K$ operator equal to the identity. For more complex cases, such as the gradient, the trace of a function, etc., analytical complications arise that prevent an immediate extension of the previous results.

The case of the operator $K = \nabla$ is of particular interest due to its applicability in viscoplastic fluid mechanics (see Section 2). However, the loss of regularity that occurs when applying the gradient leads to major difficulties in the analysis of the differentiability of the solution operator.

Optimal control problems within this context have been considered in \cite{Delosreyes2009,dlRe2010,dlReSchoen2013} with applications to viscoplastic fluids, contact mechanics and imaging. Existing results concern the characterization of Clarke stationary points and their numerical approximation. The characterization of strong stationary points is still an open problem.

\subsection{Time dependent problems}
Another important extension which has not been deeply explored yet is the optimal control of time-dependent variational inequalities of the second kind. Here, the inequality that governs the phenomena is of parabolic type and given by: Find $u \in L^2(0,T;V)$ such that
\begin{multline*} \label{eq:parabolic VI}
  \dual{\frac{\partial u}{\partial t}}{v-u} + \dual{Au}{v-u} + \beta \int_{\S} |K v| ~ds\\ - \beta \int_{\S} |K u| ~ds \geq \dual{f}{v-u}, \text{ for all } v \in V, \text{ a.e. }t \in ]0,T[,
\end{multline*}
complemented with an initial condition $u(0)=u_0 \in V.$ Existence, uniqueness and regularity of solutions for these types of inequalities have been investigated in the past \cite{DuvautLions1976}. Moreover, there are several numerical approaches for the solution of these problems (see \cite{Glowinski} and the references therein).

The corresponding optimal control problems have been approached only from a general perspective, with regularizations of global type \cite{Barbu1993}. The study of different aspects such as existence, optimality conditions, Pontryagin's maximum principle, solution regularity, sufficient conditions, etc. are still open and deserve to be investigated in the coming future. Even more so, since the obtention of results for the optimal control of time-dependent variational inequalities of the first kind has been proved to be very challenging.

\bibliographystyle{plain}
\bibliography{../../biblio}
\end{document}